\documentclass[12pt]{amsart}
\usepackage{amsfonts, amsthm, amssymb, amsmath, bm}
\usepackage{hyperref, natbib}
\hypersetup{colorlinks=false}

\usepackage{color}

\newtheorem{theo}{Theorem}[section]
\newtheorem{corollary}[theo]{Corollary}
\newtheorem{prop}[theo]{Proposition}
\newtheorem{lemma}[theo]{Lemma}

\theoremstyle{remark}
\newtheorem{remark}[theo]{Remark}

\theoremstyle{definition}
\newtheorem{defi}[theo]{Definition}

\def \tr {\operatorname{tr}}
\def \Frob {\operatorname{Frob}}

\def \std {\operatorname{std}}
\def \bfc {{\mathbf c}}

\def \bfn {{\mathbf n}}
\def \bfa {\boldsymbol{\alpha}}
\def \bff {{\mathbf f}}

\begin{document}

\title[Short intervals]{Square-root cancellation for sums of factorization functions over short intervals in function fields}
\author{Will Sawin}
\address{Columbia University, Department of Mathematics, 2990 Broadway, New York, NY 10027}

\begin{abstract} We present new estimates for sums of the divisor function, and other similar arithmetic functions, in short intervals over function fields. (When the intervals are long, one obtains a good estimate from the Riemann hypothesis.) We obtain an estimate that approaches square root cancellation as long as the characteristic of  the finite field is relatively large.  This is done by a geometric method, inspired by work of Hast and Matei, where we calculate the singular locus of a variety whose $\mathbb F_q$-points control this sum. This has applications to highly unbalanced moments of $L$-functions.\end{abstract}

\maketitle

\section{Introduction}

In this paper, we study cancellation of sums of arithmetic functions of polynomials over a finite field in ``short intervals" - that is, the sum over the set of monic polynomials with a fixed set of leading terms. For the divisor function, our estimates approach square-root cancellation over finite fields of sufficiently large characteristic.

Throughout this paper, we work over a finite field $\mathbb F_q$ of characteristic $p$. For a natural number $k$, define the $k$th divisor function $d_k(f)$ as \[ d_k(f) = \Bigl| \Bigl\{ f_1,\dots, f_k \in \mathbb F_q[T] \mid f_1,\dots f_k \textrm { monic}, \prod_{i=1}^k f_i = f \Bigr \} \Bigr|\] for a monic polynomial $f\in \mathbb F_q[T]$.

\begin{theo}[Corollary \ref{maindivisor}]\label{maindivisor-intro} For natural numbers $n,h,k$ with $h<n$ and $f$ a monic polynomial of degree $n$ in $\mathbb F_q[T]$, we have

\[ \left| \sum_{ \substack {g \in \mathbb F_q[T] \\  \deg g < h }  }  d_k(f+g ) -  { n + k-1 \choose k-1} q^{h} \right| \leq 3 { n + k-1 \choose k-1} (k+2)^{2n-h } q ^{\frac{1}{2} \left(  h+ \lfloor \frac{n}{p} \rfloor - \lfloor \frac{n-h}{p} \rfloor  +1\right)} . \]

\end{theo}

Note that the error term is $O \left(    \left(q^h \right)^{1/2}  \left(q^n\right)^{\epsilon} \right) $ as long as $p$ and $q^{  \frac{1}{\log k+2}}$ are large with respect to $1/\epsilon$. 

Note also that the asymptotic result \[ \sum_{ \substack {g \in \mathbb F_q[T] \\  \deg g < h }  }  d_k(f) = (1 +o(1)) { n + k-1 \choose k-1} q^{h}\] follows as long as $n$ goes to $\infty$ while   $\frac{h}{n} \geq \frac{ \frac{1}{p}+  2 \log_q (k+2)} { \frac{p-1}{p} + 2 \log_q (k+2) } - \epsilon $ for any $\epsilon>0$. Because we can think of $q^n$ as the size of the relevant numbers and $q^{h}$ as the length of the short interval, this cutoff is analogous to, over the integers, an asymptotic in short intervals of length $X^{\delta + o(1)}$ for $\delta=  \frac{ \frac{1}{p}+  2 \log_q (k+2)} { \frac{p-1}{p} +  2 \log_q (k+2) }$ going to $0$ as $p$ goes to $\infty$.

Over the integers, it is reasonable to expect an asymptotic formula for divisor sums in short intervals of length $X^\delta$ for any fixed $\delta>0$.  For $k=2$ this is known for $\delta>\frac{131}{416}$ \cite{Huxley}, for $k=3$ we can take $\delta>\frac{43}{96}$ \cite{Kolesnik}, for $k=4$ we can take $\delta > \frac{1}{2}$, for $k=5$ we can take $\delta>\frac{11}{20}$, and there are further estimates for greater $k$ \cite[Chapter 13]{Ivic}. These estimates all arise from bounds for the error term in the generalized Dirichlet divisor problem (as discussed in \cite*[remark on p. 2]{GaraevLucaNowak}, which improves the bound coming from Dirichlet by a logarithmic factor for $k=4$ - seemingly the only improvement known for any $k$). Given an optimal estimate for the generalized Dirichlet divisor problem, we could take any $\delta> \frac{k-1}{2k}$ (see \cite[p. 320]{Titchmarsh} for discussion of this). Theorem \ref{maindivisor-intro} gives a better range than even this optimal estimate as long as $p> (k+2)^{O(1)}$.  

One could also compare Theorem \ref{maindivisor-intro} to an asymptotic for sums of divisor functions in arithmetic progressions with modulus a large power of a small prime, as the place at $\infty$ in function fields behaves more like a small non-archimedean place than the place at $\infty$ of the rationals. For this problem, bounds in progressions of length slightly less than $X^{1/3}$ are known for $k=2$ \citep{khan_2016} \citep{LiuShparlinskiZhang}.

The variance of the sum of the divisor function in short intervals was calculated by \citet*{krrr} in the $q \to \infty$ limit. Thus, their result controls the average size of the error term for this sum, while Corollary~\ref{maindivisor} controls the worst case. A worst case bound with a savings of $\sqrt{q}$ was proven by \citet*{BBSR} for the occurrence of any polynomial factorization type in short intervals, which would in particular imply a similar estimate for the divisor function. For the M\"{o}bius function, a worst case result was proved for $q$ large with respect to the characteristic by \citet{sawinshusterman}, giving power savings, but not square-root cancellation.

Similar estimates hold for divisor-like arithmetic functions where we weigh the contribution of a given factorization of a polynomial by some function of the degrees of the factors. For arbitrary ``factorization functions", such as the M\"{o}bius and von Mangoldt functions, our results approach square root cancellation in the large $q$ limit, although they are not yet nontrivial in the large $n$ limit. Specifically, we have

 Let $\mu$ be the function field M\"{o}bius function 
 \[ \mu(f) = \begin{cases} (-1)^r & f = \prod_{i=1}^r \pi_i\textrm{ with }\pi_i \textrm{ monic, irreducible, distinct}\\ 0 &\textrm{otherwise} \end{cases}\]
and $\Lambda$ the function field von Mangoldt function.
\[ \Lambda(f) = \begin{cases}  \deg \pi & f=\pi^r, \pi \textrm{ monic, irreducible, } r \neq 0 \\ 0 & \textrm{otherwise} \end{cases}\]

\begin{theo}[Corollary \ref{mobius-mangoldt}]\label{mob-man-intro} Keep the notation of Theorem \ref{maindivisor-intro}. We have

\[ \Biggl|  \sum_{ \substack {g \in \mathbb F_q[T] \\  \deg g < h }  } \mu(f+g)  \Biggr| \leq 3 (n+2)^{2n-h} q^{\frac{1}{2}\left( h + \lfloor \frac{n}{p} \rfloor - \lfloor \frac{n-h}{p} \rfloor  +1\right)}  \]
and
\[ \Biggl|  \sum_{ \substack {g \in \mathbb F_q[T] \\  \deg g < h }  } \Lambda(f+g)    -  q^n  \Biggr| \leq 6  (n+2)^{2n-h}q^{\frac{1}{2}\left( h + \lfloor \frac{n}{p} \rfloor - \lfloor \frac{n-h}{p} \rfloor  +1\right)}  .\] \end{theo}

These bounds (and the corresponding bounds for functions associated to any irreducible representation of $S_n$) improve on \cite*{BBSR}, giving a greater power of $q$ savings, as long as $h > 5$ so that $h - ( \lfloor \frac{n}{p} \rfloor - \lfloor \frac{n-h}{p} \rfloor ) \geq \lfloor \frac{h}{2} \rfloor \geq 3$. 

One could compare also the work of \cite{CarmonEntin} counting squarefree polynomials in short intervals, although that problem is very different as squarefree polynomials are counted by a sieve and the problem is to optimize the error term in the sieve, whereas here there is no analytic approach to obtain large savings and we need to develop a completely geometric approach.

Over the integers, we expect asymptotics for von Mangoldt sums in short intervals of size $X^\delta$, and cancellation for M\"{o}bius sums in short intervals of size $X^\delta$, for any fixed $\delta>0$. For the M\"{o}bius function one could even conjecture this for intervals of size a constant power of $\log X$, but this is known to be false for the von Mangoldt function \cite{Meier}. For von Mangoldt, this is known for $\delta\geq7/12$ \cite{HuxleyPrimes, HeathBrown}. For M\"{o}bius, it is not known for any $\delta$, though we could take $\delta>1/2$ under the Riemann hypothesis - as we could for any of these problems.

\vspace{10pt}

Using the divisor sum estimate, we also obtain estimates for certain moments of $L$-functions in the large conductor limit. Specifically, let us say that a Dirichlet character mod $u^{m+1}$, $\chi: (\mathbb F_q[u]/u^{m+1})^\times \to \mathbb C^\times$ is even if it is trivial on $\mathbb F_q^\times$ and primitive if it is nontrivial on $1 + u^m \mathbb F_q[u]$.  Let $S_{m,q}$ be the set of primitive even Dirichlet characters mod $u^{m+1}$. We have $|S_{m,q}|= q^{m}- q^{m-1}$. For $\chi \in S_{m,q}$, let \[L(s,\chi) =\sum_{\substack {f \in \mathbb F_q[u] \\ \textrm{monic} \\ \gcd(fuT)=1}} \chi(f) q^{ - (\deg f) s}.\]

Then we have estimates for the moments of the $L$-functions $L(s,\chi)$ as $\chi$ varies over $S_{m,q}$.  (After applying the field automorphism $T \mapsto T^{-1}$, these give Hecke characters constant on short intervals, explaining the relation).

\begin{theo}[Corollary \ref{momentseasy}]\label{momentseasyintro} Let $m\geq 1$ and $r \geq 0$ be natural numbers and let $\alpha_1,\dots,\alpha_r$ be complex numbers with nonnegative real part. Then \[ \frac{1}{  \left|S _{m,q} \right| }  \sum_{\chi \in S_{m,q} }\prod_{i=1}^r L(1/2 + \alpha_i, \chi) \] \[= 1 + O_r \left( m^r  (r+2)^{ r(m-1)+m} q^{\frac{1}{2} \left( -m+ \lfloor \frac{ r (m-1)}{p} \rfloor - \lfloor \frac{m}{p} \rfloor + 1\right)} \right). \] \end{theo}

\begin{theo}[Corollary \ref{momentsmedium}]\label{momentsmediumintro} Let $m \geq 1, r \geq 0,$ and $s\geq 1$ be natural numbers, and let $\alpha_1,\dots,\alpha_r$ be complex numbers with nonnegative real part. Then

 \[ \frac{1}{ \left|S _{m,q} \right| }  \sum_{\chi \in S_{m,q} } \epsilon_\chi^s \prod_{i=1}^r L(1/2 + \alpha_i, \chi) = O_{r,s} \left( m^r (r+s+2)^{ (r+s) (m-1)+m} q^{\frac{1}{2}\left( -m+ \lfloor \frac{ (r+s) (m-1)}{p} \rfloor - \lfloor \frac{m}{p} \rfloor + 1\right)} \right). \] \end{theo}
 
 The notation $O_r$ and $O_{r,s}$ means an implicit constant depending only on $r$ in Theorem \ref{momentseasyintro} and depending only on $r$ and $s$ in Theorem \ref{momentsmediumintro}.

Again, note that the error terms are $O\left(\left( q^m\right)^{\epsilon-\frac{1}{2}}\right)$ as long as $p$ and $q^{\frac{1}{(r+s)\log(r+s+2)}}$ are large with respect to $\epsilon$.

To motivate \ref{momentsmediumintro}, note that these types of moments twisted by powers of the $\epsilon$-factor show up, over the integers, in proofs of the existence of $L$-functions in a given family whose critical values lie in a fixed angular sector. See \cite[\S1.7]{twoforsix} for an example of this.

\cite{wvqkr} proved estimates for the distribution of the $L$-functions in this family, and thus in particular on its moments, with a savings of $\sqrt{q}$. (Though because he proved equidistribution only in the projective unitary group and not the full unitary group, the obtained estimates would have some additional averaging in the $\alpha_i$ variables.) \cite{Tamam} proved estimates for the fourth absolute moment of the $L$-functions of a similar family, except with prime modulus, obtaining savings proportional to the degree of the modulus (i.e. logarithmic savings). \cite{AndradeYiasemides} exactly evaluated the second moment for this family and other families of Dirichlet characters with squareful modulus, and computed the fourth moment for any family of primitive Dirichlet characters with saving about the square root of the degree. While there has been significant work on the moments of $L$-functions over function fields outside the large $q$ aspect, including power savings results, we are not aware of any work dealing with arbitrarily high power moments.

Over the integers, an estimate with power savings for the moments Theorem \ref{momentseasyintro}, except to prime moduli instead of prime power, is known for $r \leq 3$. The case $r=3$ is proven in \citep{Zacharias}, and the cases $r \leq 2$ require a simplified version of the same argument. The analogue of Theorem \ref{momentsmediumintro} for $r=1$ and arbitrary $s$ follows from the approximate functional equation and Deligne's bound for Kloosterman sums, for $r=2$ and arbitrary $s$, to prime moduli, it follows from \citep*{TraceFunctionsModularForms} by the arguments of \citep[\S4]{twoforsix} (except applied to Eisenstein series rather than cusp forms). The prime power case seems less studied, but it may be possible to adapt the explicit formula of \cite{Romascavage} for the twisted second moment to a proof of the $r=2$ and arbitrary $s$ case.

\vspace{10pt}

These bounds arise from the increasingly common approach to analytic number theory over function fields that proceeds by controlling many different cohomology groups of a variety, often, though not always, using stable cohomology theory (see, e.g. \citep*{EVW} and \citep{ShendeTsimerman}). In this strategy, some desired sum or the solution of some counting problem over a polynomial ring $\mathbb F_q[T]$ is expressed as the number of $\mathbb F_q$-points of some high-dimensional variety, say $X$. Using the Grothendieck-Lefschetz fixed point formula, this is viewed as the sum of traces of Frobenius on the cohomology groups of $X$. The high-degree cohomology groups, which can give the largest contribution, by Deligne's weight theory, are calculated, shown to vanish, or otherwise controlled, while the low-degree cohomology groups merely need to have their dimension bounded. 

Our specific method is based on one applied by \cite{HastMatei} to the moments of sums of arithmetic functions, including the divisor function, in short intervals. They observed that the relevant variety is an affine cone on a projective complete intersection, and used general results that estimate the high-degree cohomology of such a complete intersection in terms of the dimension $\delta$ of its singular locus. However, because $\delta$ was relatively large, the numerical result they obtained was of the same strength as the function field Riemann hypothesis. (More precisely, as they observe in \citep{HastMatei}, the bound they prove by this cohomological method follows also from work of \cite{rodgers}, which itself relies primarily on equidistribution results of \cite{wvqkr}. However, examining the argument of \cite{rodgers}, to obtain a bound in the $q$ aspect, only the Riemann hypothesis and not equidistribution is needed. Both papers go on to prove further results, respectively on the top nontrivial cohomology of the variety and on the asymptotics in the $q$ aspect, that do not follow from the Riemann hypothesis in this way.)

Applying the same idea to sums of arithmetic functions in short intervals requires generalizing the results from affine cones to more general affine complete intersections, a straightforward \'{e}tale cohomology calculation. Once this is done, we can show by a trick involving the logarithmic derivative (Lemma~\ref{singularityimprovement}) that the dimension of the singular locus is much smaller than it was for the variety studied by Hast and Matei, and we instead obtain a result approaching square-root cancellation for the sum in short intervals (instead of the Riemann hypothesis, which gives square-root cancellation for twisted sums over all polynomials of a given degree, from which weaker bounds for sums in short intervals can be deduced).

We expect that these results can be generalized to sums in arithmetic progressions and to moments of $L$-functions over other families of Dirichlet characters. The main difficulty is finding a suitable compactification.

This research was partially conducted during the period I served as a Clay Research Fellow, and partially conducted during the period I was supported by Dr. Max R\"{o}ssler, the Walter Haefner Foundation and the ETH Zurich Foundation. I would like to thank Emmanuel Kowalski for helpful conversations,  Jon Keating for the inspiration to attack this problem, Ofir Gorodetsky for helpful comments on an earlier version of this paper, and Mark Shusterman, Philippe Michel, Valentin Blomer, and Djordje Mili\'cevi\'c for information about prior work. I would also like to thank the four anonymous referees for their many useful comments.

\section{Geometric setup}\label{s-geom}

We work over a field $\kappa$. If the characteristic of $\kappa$ is nonzero, let $p$ be the characteristic of $\kappa$, and if the characteristic of $\kappa$ is zero, let $p =\infty$. After this section, we will specialize to the case where $\kappa$ is a finite field, and so $p$ will be its characteristic.

 Let $n$ and $m$ be natural numbers with $n\geq m$. We use ${\bf c}$ to refer to a tuple $(c_1,\dots,c_m)$ of elements of $\kappa$ (or another field).

Let $X_{n,m, {\bf c}}$ be the closed subscheme of $\mathbb A^n_{\kappa}$ with variables $(a_1,\dots,a_n)$ defined by the system of $m$ equations \[ \prod_{i=1}^n (1 - u a_i) = 1 + c_1 u + c_2 u^2 + \dots + c_m u^m \mod u^{m+1} \] where $u$ is a variable. (The equations are defined by equating the coefficients of each power of $u$ on both sides.) 

Let $\overline{X}_{n, m,{\bf c}}$ be its projective closure, i.e. the subspace of $\mathbb P^n_{\kappa}$ with variables $(a_1:\dots:a_n:z)$ defined by the system of $m$ equations \[ \sum_{\substack{ S \subseteq \{1,\dots,n\} \\ |S| = r}} \prod_{i \in S} a_i  =(-1)^r c_r z^r\] for $r$ from $1$ to $m$. Let $D_{n,m,{\bf c}} =  \overline{X}- X$ be the divisor defined by $z=0$. We will drop these subscripts when the meaning is clear.

There is a natural map $\mu: \mathbb A^n \to \mathbb A^n$ that sends the point with coordinates $(a_1,\dots,a_n)$ to the point whose coordinates are the coefficients of $\prod_{i=1}^n (1 - u a_i)$, in other words, the elementary symmetric polynomials in $a_1,\dots, a_n$. This map is finite - in fact, it is the quotient map by the $S_n$ action on $\mathbb A^n$. It will often be convenient to calculate the dimension of spaces by analyzing their image under $\mu$.

\begin{lemma}\label{dimension} The dimension of $\overline{X}_{n,m,{\bf c}}$ is $n-m$. In particular, $\overline{X}_{n,m,{\bf c}}$ is a complete intersection. \end{lemma}

\begin{proof} It suffices to check that the affine cone on $\overline{X}_{n,m,{\bf c}}$ has dimension $n-m+1$, or that the subscheme of this affine cone where $z$ has any given fixed value has dimension $n-m$.  To do this, because $\mu$ is finite, it suffices to show that the image of this subscheme under $\mu$ has dimension $n-m$. Because the equations defining this subscheme are given by setting the $r$'th coefficient of $\prod_{i=1}^n (1 - u a_i) $ to equal $c_r z^r$, the image of this subscheme under $\mu$ is contained in the space where the first $m$ coordinates are set to $c_1 z,\dots, c_m z^m$. Because $m \leq n$, fixing $m$ coordinates in $n$-dimensional affine space produces an $n-m$-dimensional variety, as desired.
\end{proof}

Let $R$ be the locus in $\overline{X}$ where $|\{a_1,\dots,a_n\}| \leq m-1$.  Note that because $n \geq m > m-1$, this locus consists of points with some significant number of repetitions among the $a_1,\dots,a_n$.

\begin{lemma}\label{singularitybasic} The scheme $\overline{X}$ is smooth, and $D$ is a smooth divisor, away from $R$.  \end{lemma}

\begin{proof} At a point $(a_1:\dots:a_n:z) \in \overline{X} $, the space $\overline{X}$ is smooth and $D$ is a smooth divisor unless some linear combination of the equations defining $\overline{X}$ has vanishing derivatives with respect to the variables $a_1,\dots,a_n$.

So if $(a_1,\dots,a_n,z)$ is singular then for some nonzero $(\lambda_1,\dots,\lambda_m)$ we have \[\frac{\partial}{\partial a_j} F( a_1,\dots, a_n, \lambda_1,\dots, \lambda_r) =0 \] for all $j$ from $1$ to $n$, where  \[ F( a_1,\dots, a_n, \lambda_1,\dots, \lambda_r) = \frac{\partial}{\partial a_j}  \sum_{r=1}^m \lambda_r (-1)^r  \sum_{\substack{ S \subseteq \{1,\dots,n\} \\ |S| = r}} \prod_{i \in S} a_i \] is a linear combination of the defining equations, ignoring the constant terms $c_r z^r$. Now fix a single $j$ from $1$ to $n$. Because $F$ is the coefficient of $u^m$ in  \[ \left(\sum_{r=1}^m \lambda_r u^{m-r} \right) \left( \prod_{i=1}^n (1 - u a_i) \right),\] $\frac{\partial F}{\partial a_j}$ is the coefficient of $u^m$ in \[ \prod_{i=1}^n (1 - u a_i)  \left(\sum_{r=1}^m  \lambda_r u^{m-r} \right)   \frac { -u }{1-u a_j}.\] Let $d_i$ be such that \[ \prod_{i=1}^n (1 - u a_i)  \left(\sum_{r=1}^m \lambda_r u^{m-r} \right) (- u)= \sum_{i=1}^\infty d_i  u^i.\] Then $\frac{\partial F}{\partial a_j}$ is the coefficient of $u^m$ in \[\frac{ \sum_{i=1}^{\infty} d_i u^i }{ 1-u a_j},\] which is equal to \[\sum_{i=1}^m d_i a_j^{m-i} .\]  Letting \[g(x) = \sum_{i=1}^m d_i x^{m-i},\] we have \[ \frac{\partial F}{\partial a_j} = g(a_j).\]  Because $\left(\sum_{r=1}^m (-1)^r \lambda_r u^{m-r} \right) $ is not divisible by $u^m$,  \[\left(\sum_{r=1}^m (-1)^r \lambda_r u^{m-r} \right) \left( \prod_{i=1}^n (1 - u a_i) \right) (- u)\] is not divisible by $u^{m+1}$, so $d_1,\dots, d_m$ do not all vanish and thus $g$ is nonzero. If $(a_1,\dots,a_n,z)$ is singular, then $g(a_j)=\frac{\partial F}{\partial a_j}$ vanishes for all $j$, so the set $\{a_1, \dots a_n\}$ is contained in the set of roots of the polynomial $g$ of degree $\leq m-1$; since this polynomial is nonzero, the set $\{a_1,\dots,a_n\}$ has size $\leq m-1$, so  $(a_1,\dots,a_n,z)$ belongs to $R$.
\end{proof}

\begin{lemma}\label{singularityimprovement} The locus $R$ has dimension at most $\lfloor \frac{n}{p} \rfloor - \lfloor \frac{m}{p} \rfloor  $, and  $R \cap D$ has dimension at most $\lfloor \frac{n}{p} \rfloor - \lfloor \frac{m}{p} \rfloor  -1$. \end{lemma}

Here if $\lfloor \frac{n}{p} \rfloor - \lfloor \frac{m}{p} \rfloor=0$, for instance when $\kappa$ has characteristic $0$, so $p=\infty$, the claim that $R \cap D$ has dimension $-1$ means that $R \cap D$ is empty.

\begin{proof} It is sufficient to show that $R \cap D$ has dimension at most $\lfloor \frac{n}{p} \rfloor - \lfloor \frac{m}{p} \rfloor  -1$, because $D$ is a projective hyperplane section of $\overline{X}$ and so it follows that $R$ has dimension at most $\lfloor \frac{n}{p} \rfloor - \lfloor \frac{m}{p} \rfloor$. Fix a point $(a_1:\dots,a_n: 0)$ in $R \cap D$  and observe that \[  \frac{ \frac{\partial}{\partial u}  \prod_{i=1}^n (1- u a_i) }{  \prod_{i=1}^n (1- u a_i) } = \sum_{i=1}^n \frac{ - a_i}{ 1- u a_i} .\]

Because $z=0$, we have $\prod_{i=1}^n (1- u a_i) =1 $ modulo $u^{m+1}$, and so $\frac{\partial}{\partial u}  \prod_{i=1}^n (1- u a_i)$ is divisible by $u^m$. Hence the numerator of the left side is divisible by $u^m$ while the denominator is prime to $u$. Hence the $u$-adic valuation of the left side is at least $m$. On the other hand, examining the right side, its denominator has degree at most  $| \{ a_1,\dots,a_n\} |\leq m-1$ and it vanishes at $\infty$, hence its numerator has degree $\leq m-2$.  Thus the $u$-adic valuation of the right side is at most $m-2$ unless the right side vanishes. Because the first possibility is a contradiction, both sides vanish, and \begin{equation}\label{derivative-vanish}\frac{\partial}{\partial u}  \prod_{i=1}^n (1- u a_i)=0.\end{equation} Hence, if $\kappa$ has characteristic $p>0$, $\prod_{i=1}^n (1- u a_i) $ is a polynomial in $u^p$. Furthermore, because $(a_1: \dots : a_n : 0)$  lies in $\overline{X}$, the coefficients of $u$ up to $u^m$ of $\prod_{i=1}^n (1- u a_i) $  must vanish.

Thus $(a_1: \dots: a_n:0)$ lies in $R \cap D$ only if $\prod_{i=1}^n (1- u a_i) $ is a polynomial in $u^p$ whose coefficients of $u$ up to $u^m$ vanish. For $(a_1,\dots, a_n)$ lying in the affine cone on $R \cap D$, the same condition holds. Now let us examine the image of this affine cone under $\mu$. It consists of $n$-tuples whose $i$th coordinate vanishes $i \leq m$ or $i$ is not divisible by $p$.  The dimension of the space of such polynomials is the number of $m < i \leq n$ which are divisible by $p$, which is $\lfloor \frac{n}{p} \rfloor - \lfloor \frac{m}{p} \rfloor$. Because $\mu$ is finite, the dimension of the affine cone is at most $\lfloor \frac{n}{p} \rfloor - \lfloor \frac{m}{p} \rfloor$. Thus the dimension of $R \cap D$ is at most $\lfloor \frac{n}{p} \rfloor - \lfloor \frac{m}{p} \rfloor-1$, as desired.

If $\kappa$ has characteristic zero, \eqref{derivative-vanish} implies that $\prod_{i=1}^n (1- u a_i)=1 $ and so $a_1=\dots = a_n=0$, meaning the affine cone on $D \cap R$ has dimension $0$, $D \cap R$ is empty, and $R$ has dimension zero.
\end{proof}

The following lemma is a variant of \cite[appendix by Nicholas M. Katz, assertion (2) in proof of Theorem 1]{HooleyCompleteIntersection} and we prove it by adapting Katz's method.

\begin{lemma}\label{vanishinglemma} Let $\overline{X}$ be a complete intersection in projective space of dimension $d$ defined over $\kappa$, let $D$ be a hyperplane section in $\overline{X}$, and let $X = \overline{X} - D$. Let $Z$ be the complement of the largest open subset of $\overline{X}$ where $\overline{X}$ is smooth and $D$ is a smooth divisor. Let $\ell$ be a prime invertible in $\kappa$. Then $H^i_c ( X_{\overline{\kappa}}, \mathbb Q_\ell)=0$ for $\dim Z + d + 1< i < 2 d $.  If $\dim Z < d -1$, then $H^{2d}_c (X_{\overline{\kappa}}, \mathbb Q_\ell) = \mathbb Q_\ell(-d)$. \end{lemma}

\begin{proof}

Let us first check the claims in the case that $Z$ is empty. By convention, in this case $\dim Z =-1$, so we must show that $H^i_c ( X_{\overline{\kappa}}, \mathbb Q_\ell)=0$ for $d < i < 2d$ and $H^{2d}_c (X_{\overline{\kappa}}, \mathbb Q_\ell) = \mathbb Q_\ell(-d)$.

Recall that for a smooth complete intersection $\overline{X}$ of dimension $d$, the cohomology of $X$ in degrees $>d$ is generated by the powers of the hyperplane class, which are nonzero exactly in the even degrees from $0$ \cite[XI, Theorem 1.6 (i), (iii)]{sga7-ii}. 
To calculate $H^i_c ( X_{\overline{\kappa}}, \mathbb Q_\ell)=0$, we apply the excision \cite[XVII, (5.1.16.2)]{sga4-3} to the open set $X$ of $\overline{X}$, obtaining a long exact sequence \begin{equation}\label{excision-sequence} H^{i-1} ( \overline{X}_{\overline{\kappa}},\mathbb Q_\ell) \to H^{i-1} (D_{\overline{\kappa}} ,\mathbb Q_\ell) 
 \to  H^i(X_{\overline{\kappa}} , \mathbb Q_\ell) \to H^i ( \overline{X}_{\overline{\kappa}},\mathbb Q_\ell) \to H^i(D_{\overline{\kappa}} ,\mathbb Q_\ell) .\end{equation}
 Applying our description of the cohomology of complete intersections to both $\overline{X}$ and $D$, we see that the natural map $H^i(\overline{X}_{\overline{\kappa}},\mathbb Q_\ell)\to H^i(D_{\overline{\kappa}},\mathbb Q_\ell)$ is surjective for $i>d-1$ because its image contains all powers of the hyperplane class, it is injective for  $d< i < 2 d$ because nonzero powers of the hyperplane class are sent to nonzero powers of the hyperplane class, and it is zero in degree $i= 2 d$.  From \eqref{excision-sequence}, it follows that  $H^i_c ( X_{\overline{\kappa}}, \mathbb Q_\ell)=0$ for $d < i < 2d$ and $H^{2d}_c (X_{\overline{\kappa}}, \mathbb Q_\ell) =H^{2d}(\overline{X}_{\overline{\kappa}}, \mathbb Q_\ell)$ is generated by the $d$th power of the hyperplane class and thus is $\mathbb Q_\ell(-d)$.

Now we prove the claims for an arbitrary complete intersection. A generic such complete intersection is smooth a generic hyperplane section is smooth as well. So we may find a one-parameter family $\overline{\mathcal X}_t$ of complete intersections containing a family of hyperplanes $\mathcal D_t$ such that $\overline{\mathcal X}_0= \overline{X}$ and $\mathcal D_0=D$,  such that $\mathcal X_t$ is smooth and $\mathcal D_t$ is a smooth divisor for generic $t$.  (To do this, pick $F_1,\dots, F_r $ the defining equations of $\overline{X}$, $G_1,\dots, G_r$ the defining equations of a generic complete intersection, and let $\mathcal X_t$ be defined by $(1-t) F_1+ t G_1,\dots, (1-t) F_r + t G_r$. Then because smoothness is an open condition, smoothness at $t=1$ implies smoothness generically.)

Let $j: \overline{\mathcal X} - \mathcal D \to \overline{\mathcal X}$ be the open immersion.  We have $H^i_c ( X_{\overline{\kappa}}, \mathbb Q_\ell)= H^i(  \overline{\mathcal X}_{0,\overline{\kappa}}, j_! \mathbb Q_\ell)$ by definition.

We now use the theory of vanishing cycles from \citep[XIII, \S2.1]{sga7-ii}. This gives a functor $R \Phi$ from $\ell$-adic sheaves on a family $\mathcal X_t$ over a curve to complexes of $\ell$-adic sheaves on a special fiber $X_0$. We have the long exact sequence \citep[XIII, (2.1.8.9)]{sga7-ii}\begin{equation}\label{vanishing-cycles-sequence}[ H^* (  \overline{\mathcal X}_{0,\overline{\kappa}}, j_! \mathbb Q_\ell) \to H^*( \overline{\mathcal X}_{\overline{\eta}},  j_! \mathbb Q_\ell) \to H^* (\overline{\mathcal X}_{0,\overline{\kappa}}, R \Phi j_! \mathbb Q_\ell) .\end{equation}  

For any sheaf $\mathcal F$ on a variety of dimension $d$, $\mathcal F[d]$ is semiperverse in the sense that its support in degree $-i$ has dimension $\leq i$, so $R\Phi \mathcal F[d]$ is semiperverse by \cite[Corollary 4.6]{Illusie} (though the case we need follows also from the argument of \cite[I, Theorem 4.2(ii)]{sga7-i}), and thus $R \Phi j_! \mathbb Q_\ell [d]$ is semiperverse.

For any open immersion $j$ into a smooth variety whose complement is a normal crossings divisor, and lisse sheaf $\mathcal F$ on the open set with tame monodromy around the normal crossings divisor, $ R\Phi j_! \mathcal F$ vanishes \cite[XIII, Lemma 2.1.1]{sga7-ii}. Because $j$ is the complement of a normal crossings divisor away from $Z$, and $\mathbb Q_\ell$ is always tame, $R \Phi j_! \mathbb Q_\ell$ is supported on $Z$. 

Because $R \Phi j_! \mathbb Q_\ell [d]$ is semiperverse and supported on $Z$, its cohomology is supported in degree $\leq \dim Z$ \cite*[4.2.3]{BBDG}, so $ H^*(\overline{\mathcal X}_{0,\overline{\kappa}}, R \Phi j_! \mathbb Q_\ell)$ is supported in degree $\leq \dim Z + d$. Thus the first arrow of the long exact sequence \eqref{vanishing-cycles-sequence} is an isomorphism in degrees $> \dim Z + d + 1$. Because  \[  H^i( \overline{\mathcal X}_{\overline{\eta}},  j_! \mathbb Q_\ell)= H^i_c(\mathcal X_{\overline{\eta}}, \mathbb Q_\ell)\] and the cohomology groups are as stated for $\mathcal X_{\overline{\eta}}$, they are also as stated for $X$.

\end{proof} 

Note that $S_n$ acts on $X_{n,m,\bfc}$ by permuting the coordinates $a_1,\dots,a_n$. This preserves the defining equations because they are symmetric in these coordinates. Hence $S_n$ acts by functoriality on the cohomology groups $H^i_c(X_{n,m, \bfc, \overline{\kappa}},\mathbb Q_\ell)$.

\begin{prop}\label{vanishingstatement} Let $\ell$ be a prime invertible in $\kappa$ and let $\bfc \in \kappa^{m}$ be a fixed tuple. We have $H^i_c(X_{n,m, \bfc, \overline{\kappa}},\mathbb Q_\ell)=0$ for $n-m + \lfloor \frac{n}{p} \rfloor - \lfloor \frac{m}{p} \rfloor  +1 < i < 2 (n-m)$.

Furthermore, as long as $ \lfloor \frac{n}{p} \rfloor - \lfloor \frac{m}{p} \rfloor  +1< n-m$, we have $H^{2n-2m}_c( X_{n,m, \bfc, \overline{\kappa} },\mathbb Q_\ell) = \mathbb Q_\ell(-(n-m))$, and the action of $S_n$ on this cohomology group is trivial. \end{prop}

\begin{proof} By Lemma~\ref{dimension}, the assumptions of Lemma~\ref{vanishinglemma} are satisfied, with dimension $\dim X = n-m$. By Lemmas~\ref{singularitybasic} and~\ref{singularityimprovement}, $\dim Z \leq \lfloor \frac{n}{p} \rfloor - \lfloor \frac{m}{p} \rfloor$. Hence by Lemma~\ref{vanishinglemma}, we get vanishing in the stated degrees, and also the calculation of the top cohomology. It remains to calculate the $S_n$ action.

Let $U$ be the (open) smooth locus of $X_{n,m,\bfc}$. If $ \lfloor \frac{n}{p} \rfloor - \lfloor \frac{m}{p} \rfloor  +1< n-m$ then $U$ is dense in $X_{n,m,\bfc}$ by Lemmas \ref{dimension} and \ref{singularityimprovement}. Thus we have isomorphisms \begin{equation}\label{eq-poincare} H^{2n-2m}_c( X_{n,m, \bfc, \overline{\kappa} },\mathbb Q_\ell) = H^{2n-2m}_c( U_{ \overline{\kappa} },\mathbb Q_\ell) = (H^0 (U_{\overline{\kappa}}, \mathbb Q_\ell (n-m) )^\vee. \end{equation} by excision and Poincar\'{e} duality \cite[XIX, Lemma 2.1 and (3.2.6.2)]{sga4-3}.  By Lemma~\ref{vanishinglemma}, $H^{2n-2m}_c( X_{n,m, \bfc, \overline{\kappa} },\mathbb Q_\ell) $ is one-dimensional, so $H^0 (U_{\overline{\kappa}}, \mathbb Q_\ell (n-m) )$ is one-dimensional, and thus it must be generated by the global section $1 \in \mathbb Q_\ell$.

Because the smooth locus $U$ is stable under automorphisms, it is in particular stable under $S_n$, so $S_n$ also acts on $U$. The global section $1$ is clearly $S_n$-invariant, so $S_n$ acts trivially on $H^0 (U_{\overline{\kappa}}, \mathbb Q_\ell (n-m) )$. Because the isomorphisms in \eqref{eq-poincare} are canonical, they commute with the automorphisms in $S_n$, so $S_n$ acts trivially on $H^{2n-2m}_c( X_{n,m, \bfc, \overline{\kappa} },\mathbb Q_\ell)$. \end{proof}
%
%
%
%
%

\section{Factorization functions}\label{s-ff}

We now explain the key definition of a factorization function on monic polynomials of degree $n$. We associate one such function to each representation $\pi$ of $S_n$. Throughout, we fix a prime $\ell$.

\begin{defi}\label{fac-fun} Fix $\pi$ a representation of $S_n$, not necessarily irreducible, defined over $\mathbb Q_\ell$. 

For $f$ a polynomial of degree $n$ over $\mathbb F_q$, let $V_f$ be the free $\mathbb Q_\ell$ vector space generated by tuples $(a_1,\dots,a_n) \in \overline{\mathbb F}_q$ with $\prod_{i=1}^n (T-a_i)=f$. Then $V_f$ admits a natural $S_n$ action by permuting the $a_i$, as well as an action of $\Frob_q$ by sending $(a_1,\dots,a_n) $ to $(a_1^q,\dots, a_n^q)$, which commute.

Let $\Frob_q$ act trivially on $\pi$. This commutes with the $S_n$ action.

Let \[F_\pi (f) = \tr (\Frob_q, (V_f \otimes \pi)^{S_n} )\] where the action of $\Frob_q$ on $(V_f \otimes \pi)^{S_n}$ is well-defined because the $\Frob_q$ and $S_n$-actions on $V_f \otimes \pi$ commute.  \end{defi}

\begin{lemma}  In the case when $f$ is squarefree, we have $F_\pi(f) = \tr(\sigma_q, \pi)$ where $\sigma_q$ is an element of the conjugacy class of $\Frob_q$ acting on the roots of $f$. \end{lemma} 

\begin{proof} For such an $f$, we have an isomorphism $V_f \cong \mathbb Q_\ell[S_n]$.
  
  This is obtained by fixing a factorization $f =\prod_{i=1}^n (T-a_i)$ and sending $\sigma$ in $S_n$ to  $(a_{\sigma(1)},\dots, a_{\sigma(n)} )$. Having does this, the $S_n$ action on $V_f$ corresponds to the right $S_n$-action (by the inverse) on $\mathbb Q_\ell[S_n]$ , and the action of $\Frob_q $  corresponds to the left action of the unique permutation $\sigma_q \in  S_n$ with $a_{\sigma(i)} = a_i^q$. The conjugacy class of $\sigma_q$ is the usual conjugacy class of Frobenius on $S_n$.

This gives an isomorphism $(V_f \otimes \pi)^{S_n} \cong( \mathbb Q_\ell [S_n] \otimes \pi)^{S_n}\cong \pi$. Under this isomorphism, the action of $\Frob_q$ on $(V_f \otimes \pi)^{S_n} $ is sent to the action of $\sigma_q$ on $\pi$, because the isomorphism $\pi \to  (\mathbb Q_\ell [S_n] \otimes \pi)^{S_n}$ sends $v \in \pi$ to  $ \sum_{\sigma \in S_n} [\sigma] \otimes \sigma^{-1} (v) $  and so it sends $\sigma_q( v)$ to \[ \sum_{\sigma \in S_n} [\sigma] \otimes \sigma^{-1}( \sigma_q(v) ) = \sum_{\sigma \in S_n} [\sigma_q \sigma] \otimes \sigma^{-1}(v ) = \Frob_q \Bigl( \sum_{\sigma \in S_n} [ \sigma] \otimes \sigma^{-1}(v ) \Bigr).\]

 Hence $F_\pi(f) = \tr(\Frob_q,   (V_f \otimes \pi)^{S_n} )= \tr( \sigma_q, \pi)$.\end{proof}

  The function $F_\pi$ is from several perspectives the most natural extension of the character of $\pi$, evaluted on the conjugacy class of $\Frob_q$, from squarefree polynomials to all polynomials. For instance, it agrees with some definitions previously the literature:

\begin{remark}Definition \ref{fac-fun} is a special case of \cite[Definition 1.1 and Section 1.3]{Gadish}. In fact, thee first step in the proof of Proposition \ref{cohomologyestimate} below is essentially a special case of \cite[Theorem A]{Gadish}.
\end{remark}

\begin{remark} The functions $F_\pi$  span the arithmetic functions of von Mangoldt type in the sense of \cite[Definition 4.3]{HastMatei}. \end{remark}

Furthermore, many natural arithmetic functions can be expressed as $F_\pi$ or as linear combinations of them.

\begin{lemma}\label{fac-Mobius}Let $\operatorname{sign}$ be the sign representation of $S_n$. For $f$ monic of degree $n$, we have\[ \mu(f) = (-1)^{n} F_{\operatorname{sign}}.\]

\end{lemma}

\begin{proof}  First we assume that $f$ is not squarefree. In this case $\mu(f)=0$, so it suffices to show that $(V_f \otimes \operatorname{sign} )^{S_n}=0$. Because $f$ has repeated roots, each basis vector of $V_f$ corresponds to a tuple $a_1,\dots ,a_n$ with  $a_i =a_j$ for some $1 \leq i< j \leq n$.  A transposition in $S_n$ that switches $i$ and $j$ fixes this basis vector but acts as $-1$ on $\operatorname{sign}$. Thus the coefficient of this basis vector in any $S_n$-invariant element of $V_f \otimes \operatorname{sign} $ is equal to minus itself, and vanishes. Because every coefficient of every $S_n$-invariant element of $S_n$-invariant element of $V_f \otimes \operatorname{sign} $ vanishes, $V_f \otimes \operatorname{sign} =0$.

Now assume that $f$ is squarefree. Then $F_\pi$ is simply the sign of Frobenius in $S_n$, which is $(-1)$ to the number of even cycles. Multiplying by $(-1)^n$, we obtain $(-1)$ raised to the number of all cycles, which is $(-1)$ raised to the number of prime factors of $f$, which by definition is the M\"{o}bius function.  \end{proof}

\begin{lemma}\label{fac-von} For $f$ monic of degree $n$, we have \[ \Lambda(f) = \sum_{i=0}^{n-1} (-1)^i F_{\wedge^i ( \std )} =   F_{\oplus_{i \textrm  { even}} \wedge^i (\std ) } -   F_{\oplus_{i \textrm { odd}} \wedge^i (\std ) }.\]
 \end{lemma} 
 
This identity was explained to me by Vlad Matei.
 
 \begin{proof} First observe that $(V \otimes \wedge^i {\std})^{S_n}$ is $\wedge^i$ of the vector space of functions on the roots of $f$ (ignoring multiplicity) that sum to zero. (This follows by induction from the fact that $(V \otimes \wedge^i ({\std} + \mathbb Q_\ell))^{S_n}$ is the space of antisymmetric functions on $n$-tuples of roots of $f$ and thus is $\wedge^i$ of the space of functions on the roots of $f$).
 
 So $\sum_{i} (-1)^i F_{\wedge^i (\std )}(1)$ is the characteristic polynomial of Frobenius on the vector space of functions on the roots of $f$ that sum to $0$. Hence $\sum_{i} (-1)^i F_{\wedge^i (\std )}(1)$  vanishes if Frobenius fixes any function on the roots of $f$ that sums to $0$. This happens if Frobenius has more than one orbit on the roots of $f$. In other words, $\sum_{i} (-1)^i F_{\wedge^i (\std )}(1)$ vanishes when $f$ is not a prime power. When $f$ is the $n/k$th power of a prime of degree $k$, this vector space has dimension $(k-1)$ and the eigenvalues of Frobenius are all the nontrivial $k$th roots of unity, so the value of the characteristic polynomial is the evaluation of $(1-x^k )/(1-x)$ at $1$, which is $k$. On the other hand, the von Mangoldt function is $ k $ in this case.

The last identity then follows from the additivity of $F_\pi$ in the representation $\pi$, which follows from the additivity of the tensor product, cohomology, $S_n$-invariants, and trace operations used to define it.

 \end{proof}
 
 Rather than describing the divisor function as a factorization function directly, it will be more useful to express a certain variant of it:
  
 \begin{defi} For natural numbers $n_1,\dots, n_k$ with $\sum_{i=1}^k n_i =n$, let $d_k^{ (n_1,\dots,n_k)}(f)$ be the function that takes a monic polynomial $f \in \mathbb F_q[T]$ of degree $n$ to the number of tuples $g_1,\dots,g_k$ of monic polynomials in $\mathbb F_q[T]$, with $g_i$ of degree $n_i$ for all $i$, such that $\prod_{i=1}^k g_i = f$. \end{defi}

It is clear from the definitions that, for $f$ monic of degree $n$, \begin{equation}\label{divisor-relation} d_k(f) = \sum_{\substack{n_1,\dots,n_k \in \mathbb N \\ \sum_{i=1}^k n_i =n} } d_k^{(n_1,\dots,n_k)}(f)  \end{equation}
The individual terms $d_k^{(n_1,\dots,n_k)}(f)  $ will be relevant in our proofs involving $L$-functions.

 \begin{lemma}\label{fac-div} For $f$ monic of degree $n$ have \[ F_{ \operatorname{Ind}_{S_{n_1} \times \dots \times S_{n_k}}^{S_n} \mathbb Q_\ell} (f) = d_k^{ (n_1,\dots,n_k)}(f).\]  \end{lemma}

\begin{proof}Observe that $F_{ \operatorname{Ind}_{S_{n_1} \times \dots \times S_{n_k}}^{S_n}} (f)$ is the trace of $\Frob_q$ on  $(V_f \otimes \operatorname{Ind}_{S_{n_1} \times \dots \times S_{n_k}}^{S_n} \mathbb Q_\ell)^{S_n} = V_{f}^{ S_{n_1} \times \dots \times S_{n_k}}$. Now $V_f^{ S_{n_1}\times \dots \times S_{n_k}}$ is the free vector space on the $S_{n_1}\times \dots \times S_{n_k}$-orbits of factorizations of $f$ into linear factors, which are simply the factorizations of $f$ into factors of degree $n_1,\dots,n_k$. The Frobenius element acts by permuting these, hence its trace is equal to the number of Frobenius-fixed factorizations, which is the number of factorizations defined over $\mathbb F_q$, as desired. \end{proof}

\section{Proofs of the main theorems}

In this section, we work in the geometric setting of Section \ref{s-geom}, but specialize the base field $\kappa$ to a finite field $\mathbb F_q$. We fix a prime $\ell$ invertible in $\mathbb F_q$.

To combine the geometric setup of Section \ref{s-geom} with the factorization functions discussed in Section \ref{s-ff}, which deal with monic polynomials, we introduce a chance of variables. The key point is that the coefficients of the polynomial $\prod_{i=1}^n (1- u a_i)$ in $u$, in ascending order, are equal to the coefficients of the monic polynomial $\prod_{i=1}^n (T - a_i)$, in descending order. Thus the map $\mu: \mathbb A^n \to \mathbb A^n$ defined in Section \ref{s-geom} as the map that sends $(a_1,\dots,a_n)$ to the coefficients of $\prod_{i=1}^n (1- u a_i)$ is equal to the map that sends $(a_1,\dots, a_n)$ to the coefficients of $\prod_{i=1}^n (T -a_i)$, as long as we interpret the coefficients in ascending and descending order respectively. 

For $\bfc \in \mathbb F_q^m$ define \[ \mathcal I_{\bfc} = \{ f \in \mathbb F_q[T] | f = T^n + c_1 T^{n-1} + \dots + c_m T^{n-m} + \dots \}  .\] We think of $\mathcal I_{\bfc}$ as a short interval in $\mathbb F_q[T]$.

Then $X_{n,m,\bfc}(\mathbb F_q)$ is the set of $a_1,\dots,a_n \in \mathbb \mathbb F_q$ with $\prod_{i=1}^n (T -a_i) \in \mathcal I_{\bfc}$. Furthermore, because $\mu$ is the quotient map under $S_n$, the quotient $ X_{n,m,\bfc}/S_n$ is the image of $X_{n,m,\bfc}$ under $\mu$, so $(X_{n,m,\bfc}/S_n) (\mathbb F_q) $ is exactly the set of monic polynomials $\mathcal I_c$.

(The reason for this change of variables is that working in the $u$ variable simplifies the calculations in Section \ref{s-geom}, while the $T$ variable is needed here.)

We will now prove a bound for sums of factorization functions. The error term will depend on the following quantity, which we will bound afterwards:

\begin{defi} For $\pi$ a representation of $S_n$, let $B(\pi) = \sum_{i=0}^{2 \dim X} \dim \left( \left( H^i_c (X_{n,m, \bfc}, \mathbb Q_\ell) \otimes \pi \right)^{S_n} \right).$\end{defi}

\begin{prop}\label{cohomologyestimate} Let $n>m$ be natural numbers. Let $c_1,\dots, c_m$ be elements of $\mathbb F_q$ and let $\pi$ be a representation of $S_n$. Then,

\[ \Biggl| \sum_{ f \in \mathcal I_{\bfc} } F_{\pi}(f) - q^{n-m} \dim \left(\pi^{S_n}\right) \Biggr| \leq B(\pi) q^{\frac{1}{2}\left( n-m + \lfloor \frac{n}{p} \rfloor - \lfloor \frac{m}{p} \rfloor  +1\right)}.\] \end{prop}

\begin{proof} Let $\rho: X_{n,m, \bfc} \to X_{n,m, \bfc}  / S_n$ be the projection map. Equivalently, $\rho$ is the restriction of $\mu$ to $X_{n,m , \bfc}$. Then \begin{equation}\label{bring-sn-inside} \left( H^i_c (X_{n,m, \bfc}, \mathbb Q_\ell) \otimes \pi \right)^{S_n} = \left(H^i_c( X_{n,m, \bfc}/ S_n,   \rho_* \mathbb Q_\ell ) \otimes \pi \right)^{S_n}  = H^i_c( X_{n,m, \bfc}/ S_n,  ( \rho_* \mathbb Q_\ell \otimes \pi ) ^{S_n} ) .\end{equation}

Here the first identity follows from the Leray spectral sequence with compact supports \cite[XVII, (5.1.8.2)]{sga4-3} applied to $\rho$, using the fact that $\rho$ is finite and so its higher cohomology vanishes, and the second uses the fact that \'{e}tale cohomology is an additive functor and finite group actions are semisimple in characteristic zero so \'{e}tale cohomology commutes with invariants under a finite group action.

By the Grothendieck-Lefschetz fixed point formula \begin{equation}\label{gl} \sum_i (-1)^{i} \tr( \Frob_q, H^i_c( X_{n,m, \bfc}/ S_n,  ( \rho_* \mathbb Q_\ell \otimes \pi ) ^{S_n} )  ) = \sum_{ x \in (X_{n,m, \bfc}/ S_n )(\mathbb F_q)} \tr(\Frob_q,  ( \rho_* \mathbb Q_\ell \otimes \pi ) ^{S_n}_x ).\end{equation}

We have seen earlier that $(X_{n,m, \bfc}/ S_n)(\mathbb F_q)$ is the short interval $\mathcal I_{\bfc}$. Thus to check that
\begin{equation}\label{geometry-equals-nt} \sum_{ x \in X_{n,m, \bfc}/ S_n (\mathbb F_q)} \tr(\Frob_q,  ( \rho_* \mathbb Q_\ell \otimes \pi ) ^{S_n}_x ) =  \sum_{ f \in \mathcal I_{\bfc} } F_{\pi}(f)  \end{equation}
it suffices to show that, for the point $x$ corresponding to the polynomial $f$,
\[ \tr\left(\Frob_q,  ( \rho_* \mathbb Q_\ell \otimes \pi ) ^{S_n}_x \right)  = F_\pi(f). \]

The stalk of $\rho_* \mathbb Q_\ell$ at $f$ is the free vector space generated by the fiber of $\rho$ over $f$. This free vector space is canonically isomorphic to $V_f$ because the fiber over $\rho$ consists of tuples $(a_,\dots, a_n ) \in \overline{\mathbb F}_q$ with $\prod_{i=1} (T-a_i)=f$. The natural actions of $S_n$ and $\Frob_q$ on this vector space are by permutation on the fiber of $\rho$, which means that $S_n$ acts by permuting the $a_i$s and $\Frob_q$ acts by raising each $a_i$ to the $q$th power. This matches the actions we have defined for $V_f$.

Thus we obtain \eqref{geometry-equals-nt}. Combining \eqref{geometry-equals-nt} with \eqref{bring-sn-inside} and \eqref{gl}, we get \[ \sum_{f \in \mathcal I_{\bfc} } F_\pi(f) =  \sum_i (-1)^i \tr\left( \Frob_q, \left( H^i_c (X_{n,m, \bfc}, \mathbb Q_\ell) \otimes \pi \right)^{S_n} \right)  .\]
\
We have $H^{2(n-m)}(X_{n,m, \bfc}, \mathbb Q_\ell)= \mathbb Q_\ell( - (n-m))$ with trivial $S_n$ action. So the contribution of $i=2(n-m)$ is  $q^{n-m} \pi^{S_n}$.

For each other $i$, by \cite[Theorem 1]{weilii}, all eigenvalues of $\Frob_q$ on $H^i_c$ are bounded by $q^{i/2}$, so the contribution of $H^i_c$ is bounded by $q^{i/2} \dim \left( \left( H^i_c (X_{n,m, \bfc}, \mathbb Q_\ell) \otimes \pi \right)^{S_n} \right)$. By Proposition~\ref{vanishingstatement}, these cohomology groups vanish unless $i \leq n-m + \lfloor \frac{n}{p} \rfloor - \lfloor \frac{m}{p} \rfloor  +1$. Summing over all $i$, we obtain the stated bound. \end{proof}

\begin{prop}\label{bettibound} Let $n_1,\dots,n_k\geq 0$ be natural numbers summing to $n$ and let $\pi$ be a subrepresentation of $\operatorname{Ind}_{S_{n_1} \times \dots \times S_{n_k}}^{S_n} \mathbb Q_\ell$. Then for any $\bfc$, we have $B(\pi) \leq 3 (k+2)^{n+m}$. \end{prop}

\begin{proof} Using Frobenius reciprocity and the projection formula, we have \[  \left( H^i_c (X_{n,m, \bfc}, \mathbb Q_\ell) \otimes\operatorname{Ind}_{S_{n_1} \times \dots \times S_{n_k}}^{S_n} \mathbb Q_\ell \right)^{S_n} \] \[ = \left( H^i_c (X_{n,m, \bfc}, \mathbb Q_\ell)\right)^{ S_{n_1} \times \dots \times S_{n_k}} = H^i_c ( X_{n,m, (c_1,\dots, c_m)}/ ( S_{n_1} \times \dots \times S_{n_k}), \mathbb Q_\ell).\]

Now $X_{n,m,( c_1,\dots, c_m)}/ (S_{n_1} \times \dots \times S_{n_k})$ is the moduli space of $k$-tuples of monic polynomials, of degrees $n_1,\dots, n_k$ whose product has leading terms $T^n + c_1 T^{n-1} + \dots + c_m T^{n-m}$. This identification can be obtained by sending $(a_1,\dots, a_n) $ to \[ \Bigl(\prod_{i=1}^{n_1}(T-a_i),  \prod_{i=n_1+1}^{n_1+2} (T-a_i),\dots, \prod_{i=n-n_k-1}^{n} (T-a_i)\Bigr).\]

Thus $X_{n,m,( c_1,\dots, c_m)}/ (S_{n_1} \times \dots \times S_{n_k})$ is a variety defined by $m$ equations of degree at most $k$ in $n$ variables.

The statement when $\pi= \operatorname{Ind}_{S_{n_1} \times \dots \times S_{n_k}}^{S_n} \mathbb Q_\ell$ then follows from \cite[Theorem 12]{katzbetti} (specifically, we take $s=0, f=0$, $N=n$, $r=m$, and $F_1,\dots, F_r$ to be these $m$ equations).

For $\pi$ a subrepresentation of the induced representation, it follows when we note subrepresentations of a finite group representation are automatically summands, hence the cohomology of $\pi$ is a summand of the cohomology of the induced representation, and thus its Betti numbers are no greater.\end{proof}

\begin{corollary}\label{bettibound2} Let $\pi$ be a subrepresentation of the regular representation of $S_n$. Then $B(\pi) \leq 3 (n+2)^{n+m}$. In particular, this holds if $\pi$ is any irreducible representation of $S_n$. \end{corollary}

\begin{proof} The claim on subrepresentations of the regular representation follows from Proposition~\ref{bettibound} taking $k=n, n_1,\dots,n_k=1$. That every irreducible representation is a subrepresentation of the regular representation is standard in representation theeory. \end{proof}

\begin{theo}\label{main} For natural numbers $n,m$ with $n \geq m$, a finite field $\mathbb F_q$ of characteristic $p$,  $\bfc\in \mathbb F_q^m$, and natural numbers $n_1,\dots, n_k$ with $\sum_{i=1}^k n_i =n$,

\[ \Biggl| \sum_{f \in \mathcal I_{\bfc} } d_k^{(n_1,\dots,n_k)} (f) -  q^{n-m} \Biggr| \leq 3 (k+2)^{n+m}q^{\frac{1}{2}\left(  n-m + \lfloor \frac{n}{p} \rfloor - \lfloor \frac{m}{p} \rfloor  +1\right)} . \]

\end{theo} 

\begin{proof} If $n>m$, this follows from Propositions~\ref{cohomologyestimate} and~\ref{bettibound}, recalling from Lemma \ref{fac-div} that $F_{ \operatorname{Ind}_{S_{n_1} \times \dots \times S_{n_k}}^{S_n} \mathbb Q_\ell} (f) = d_k^{ (n_1,\dots,n_k)}(f)$. 

In the case $n=m$, note that the only $f$ appearing in the sum on the left is $T^n+ c_1 T^{n-1} + \dots + c_m $, whose $d_k^{n_1,\dots,n_k}$ is between $0$ and  $k^{n}$, and $q^{n-m}$ is $1$, so the left side is at most $k^{n}-1$, and the right side is $3 (k+2)^{2n} \sqrt{q}$, so the inequality holds.

 \end{proof}

\begin{corollary}\label{maindivisor} For natural numbers $n,m$ with $n\geq m$, a finite field $\mathbb F_q$ of characteristic $p$,  and $c_1,\dots,c_m \in \mathbb F_q$, 

\[ \Biggl| \sum_{f \in \mathcal I_{\bfc} } d_k(f) -  { n + k-1 \choose k-1} q^{n-m} \Biggr| \leq 3 { n + k-1 \choose k-1} (k+2)^{n+m}q^{\frac{1}{2}\left( n-m + \lfloor \frac{n}{p} \rfloor - \lfloor \frac{m}{p} \rfloor  +1\right)} . \]

\end{corollary}

\begin{proof} This follows from Theorem \ref{main} by summing over all possible $n_1,\dots,n_k$. \end{proof}

We can prove similar results for the M\"{o}bius and von Mangoldt functions, using Corollary \ref{bettibound2}.

\begin{corollary}\label{mobius-mangoldt} We have

\begin{equation}\label{mob-cor} \Biggl| \sum_{ f \in \mathcal I_{\bfc} } \mu(f)  \Biggr| \leq 3 (n+2)^{n+m}q^{\frac{1}{2}\left(  n-m + \lfloor \frac{n}{p} \rfloor - \lfloor \frac{m}{p} \rfloor  +1\right)}  \end{equation}
and
\begin{equation}\label{man-cor} \Biggl| \sum_{ f \in \mathcal I_{\bfc} } \Lambda(f)  - q^n  \Biggr| \leq 6  (n+2)^{n+m}q^{\frac{1}{2}\left( n-m + \lfloor \frac{n}{p} \rfloor - \lfloor \frac{m}{p} \rfloor  +1\right)}  \end{equation}
\end{corollary}

\begin{proof} \eqref{mob-cor} follows from Theorem \ref{main} when we recall from Lemma \ref{fac-Mobius} that $\mu$ is $(-1)^n F_\pi$ for $\pi$ the sign representation, using Corollary \ref{bettibound2} to bound $B(\pi)$. 

For \eqref{man-cor}, we use almost the same strategy. First, we recall from Lemma \ref{fac-von} that
\[ \Lambda =  \sum_{i} (-1)^i F_{\wedge^i (\std )} =  F_{\oplus_{i \textrm  { even}} \wedge^i (\std ) } -   F_{\oplus_{i \textrm { odd}} \wedge^i (\std ) }.\] 
Next use the fact that both $\oplus_{i \textrm { even}} \wedge^i (\std )$ and $\oplus_{i \textrm { odd}} \wedge^i (\std )$ are sums of distinct irreducible representations, hence subrepresentations of the regular representation, so Corollary \ref{bettibound2} may be applied to each in term. The main term in \eqref{man-cor} comes from the trivial representation $\wedge^0(\std)$, as none of the other representations appearing have $S_n$-invariants.

\end{proof}

Theorems \ref{maindivisor-intro} and \ref{mob-man-intro} stated in the introduction follow by setting $m=n-h$ and $c_1,\dots,c_m$ to be the coefficients of $T^{n-1},\dots, T^{n-m}$ in $f$. Indeed, the set $\mathcal I_{\bfc}$ of polynomials whose leading terms are $T^n + c_1 T^{n-1} + \dots + c_m T^{n-m}$ is exactly the set of polynomials $f+g$ with $\deg g< n-m = h$. We obtain the bounds stated in the introduction by substituting $n-h$ for $m$ in the bounds stated in this section.

\section{Moments of L-functions}

Recall that we say $\chi:( \mathbb F_q[u]/ u^{m+1} )^\times \to \mathbb C^\times $ is primitive and even if it is trivial on $\mathbb F_q^\times$ but nontrivial on $1+ u^m \mathbb F_q$, and we let $S_{m,q}$ be the set of all primitive even characters. We have the $L$-function \begin{equation} L(s,\chi) = \sum_{ \substack{  f \in \mathbb F_q[u] \\ \textrm{monic} \\ \gcd(f,u)=1}} \chi(f) q^{ - s \deg f} .\end{equation}

Alternately, we can transform $\chi$ to a character of monic polynomials in $\mathbb F_q[T]$ by setting $\psi(f) = \chi ( f(u^{-1}) u^{\deg f} )$ for $f \in \mathbb F_q[T]$. Let $S'_{m,q}$ be the set of characters $\psi$ arising from $\chi \in S_{m,q}$ by this formula. Here the natural $L$-function is \begin{equation} L(s,\psi)  = \sum_{ \substack{ f \in \mathbb F_q[T] \\ \textrm{monic}}} \psi(f) q^{- s \deg f }. \end{equation} Such a character $\psi$ depends only on the leading $m+1$ coefficients of $f$.

\begin{lemma}\label{chi-psi-comparison} For $\chi$ a primitive even character of  $( \mathbb F_q[u]/ u^{m+1} )^\times$ and $\psi(f) = \chi ( f(u^{-1}) u^{\deg f} )$, we have \begin{equation}  L(s,\psi) =\frac{1}{ 1 - q^{-s}}  L(s,\chi) .\end{equation} \end{lemma}

\begin{proof} We can see this from the substitution $\sigma: f \mapsto  f(u^{-1}) u^{\deg f} / a_f$, where $a_f$ is the lowest nonzero coefficient of $f$. We observe that every monic polynomial in $T$ is sent by $\sigma$ to a monic polynomial prime to $u$, and that the inverse image of a monic polynomial of degree $d$, prime to $u$ under $\sigma$ consists of a polynomial of degree $d$ and one of each greater degree (its multiples by power of $T$). Thus, these two $L$-functions agree up to a factor of $1/(1-q^{-s})$ which accounts for the multiples of a given polynomial by powers of $T$.

Alternately, viewing $\psi$ and $\chi$ as idele class characters of the fields $\mathbb F_q(u)$ and $\mathbb F_q(T)$ respectively, and defining the $L$-functions using the idele class group, this identity arises from the field isomorphism $u=T^{-1}$, with $\frac{1}{1-q^{-s}}$ being the local factor at $\infty$. \end{proof} 

\begin{remark} While these two families of $L$-functions behave identically over function fields, they suggest different analogies over number fields.

The $L$-functions of Dirichlet characters $\chi$ of $\mathbb F_q[u]$ appear similar to the $L$-functions of Dirichlet characters of the integers modulo a large power of a small prime.

The characters $\psi$ of $\mathbb F_q[T]$ depending only on the leading coefficients of $f$ are ramified at $\infty$, thus are an analogue of the characters $n \mapsto n^{it}$ of the integers, which are ``ramified at $\infty$" in some sense, and so its moments might appear like the moments of the Riemann zeta function. However, the fact that the characters $\psi$ are split at the prime $T$ makes this more difficult, as if we fixed $t$ to lie in the arithmetic progression $\frac{ 2 \pi i}{\log 2} \mathbb Z$. The moments of $\zeta$ on arithmetic progressions have been studied over the integers \citep{LiRadziwill}, though not as much as Dirichlet characters.

The comparison to Dirichlet characters probably provides the best analogy for the difficulty of the moments of this family by traditional analytic means, because the infinite place can behave differently over number fields and function fields.
\end{remark}

Using the divisor function estimates above, we immediately obtain results on moments of $L$-functions for this family of characters. To do this, it is most convenient to work with the formulation in terms of $\psi$.

  We first describe the basic properties of these $L$-functions:

\begin{lemma}\label{L-polynomial} For $\psi \in S_{m,q}$, $L(s,\psi)$ is a polynomial in $q^{-s}$ of degree $m-1$ and satisfies the functional equation \begin{equation} L(s,\psi) = \epsilon_\psi  q^{(m-1) (s-1/2)} L(1-s, \overline{\psi} ) \end{equation} where \begin{equation}\label{root-number-formula}\epsilon_\psi =  q^{- (m-1)/2} \sum_{ \substack{ f\in \mathbb F_q[T] \\ \textrm{monic} \\ \deg f= m-1}} \psi(f)  .\end{equation} \end{lemma}

This is a special case of the standard \cite[Theorem 9.24A]{Rosen} (or \cite[Proposition 3.9]{AndradeYiasemides} for the analogue stated in terms of $\chi$) but we review its short proof here.

\begin{proof} To check that $L(s,\psi)$ is a polynomial, it suffices to check that the coefficient \[ \sum_{ \substack{ f \in \mathbb F_q[T] \\ \textrm{monic} \\ \deg f= d}} \psi(f) \] of $q^{-ds}$ vanishes for $d \geq m$. This follows by orthogonality of characters because $\psi$ is nontrivial and each possible set of leading $m+1$ terms occurs equally often among the monic $f$.

For the functional equation, there are two approaches. First, we could use the result of \citet{WeilExponential} that the polynomial has degree exactly $m-1$ and all roots of size $\sqrt{q}$, which implies the functional equation. Alternately, the functional equation can be proved directly, by Fourier analysis, and it implies that the degree of the $L$-function is $m-1$. \end{proof}

In the proofs of the results below, it will be convenient to use big $O$ notation for polynomials, where $f=g + O(T^n)$ denotes that $f-g$ is a polynomial of degree at most $n$.

\begin{lemma}\label{average-of-characters} For $f$ monic of degree $n$,

\[ \sum_{\psi \in S'_{m,q} }\psi(f) = \begin{cases} q^m - q^{m-1} &  \textrm{if } f= T^n + O(T^{n-m-1}) \\  -q^{m-1} & \textrm{if }f\neq  T^n + O(T^{n-m-1})\textrm{ but } f = T^n + O (T^{n-m}) \\ 0 & \textrm{otherwise}. \end{cases} .\]

\end{lemma}

\begin{proof} This follows from the identities
\[ \sum_{\substack{ \chi: ( \mathbb F_q[u]/u^{m+1})^\times \to \mathbb C^\times \\ \textrm{even} }} \chi(f(u)^{-1} u^{\deg f} ) =  \begin{cases} q^m  &  \textrm{if } f= T^n + O(T^{n-m-1}) \\ 0 & \textrm{otherwise}. \end{cases} \] and 
\[ \sum_{\substack{ \chi: (\mathbb F_q[u]/u^{m})^\times \to \mathbb C^\times \\ \textrm{even}  }} \chi(f(u)^{-1} u^{\deg f} )  =  \begin{cases} q^{m-1}  &  \textrm{if } f= T^n + O(T^{n-m}) \\ 0 & \textrm{otherwise}. \end{cases} \] both of which follow immediately from orthogonality of characters. \end{proof}

\begin{corollary}\label{momentseasy} Let $\alpha_1,\dots,\alpha_r$ be complex numbers with nonnegative real part. \begin{equation}\label{momentseasy-main} \begin{split} &\frac{1}{ \left| S'_{m,q} \right|  }  \sum_{\psi \in S'_{m,q} }\prod_{i=1}^r L(1/2 + \alpha_i, \psi) \\ = &\prod_{i=1}^r \frac{ 1}{ 1- q^{ -(1/2 + \alpha_i) } }+ O \left( m^r  (r+2)^{ r(m-1)+m} q^{\frac{1}{2}\left( -m + \lfloor \frac{ r (m-1)}{p} \rfloor - \lfloor \frac{m}{p} \rfloor + 1\right)} \right). \end{split} \end{equation} \end{corollary}

To obtain Theorem \ref{momentseasyintro} from \eqref{momentseasy-main}, we use Lemma \ref{chi-psi-comparison} to replace $L(1/2+\alpha_i,\psi)$ with $L(1/2+\alpha_i,\chi)$, which cancels the  $ \frac{ 1}{ 1- q^{ -(1/2 + \alpha_i) }} $ factors in the main term and introduces an $1- q^{ -(1/2 + \alpha_i)  } = O(1)$ factor in the error term.

 \begin{proof} We have \begin{equation}\label{momentseasy-first} \prod_{i=1}^r L(1/2 + \alpha_i, \psi) = \sum_{n_1,\dots, n_r \in \{0,\dots, m-1\}}   \prod_{i=1}^r q^{ - n_i (1/2 + \alpha_i) }  \sum_{\substack{ f_1,\dots, f_r \in \mathbb F_q[T] \\ \textrm{monic} \\ \deg(f_i) = n_i }} \psi( f_1f_2 \dots f_r).\end{equation}
 (We use Lemma \ref{L-polynomial} to deduce that it suffices to sum over $n_i \leq m-1$.)
 
 We introduce some notation to simplify this sum. Let $\bfn = (n_1,\dots,n_r)$, $n =\sum_{i=1}^r n_i$., $\bfa = (\alpha_1,\dots,\alpha_r)$,  and $Q (\bfn, \bfa) =  \prod_{i=1}^r q^{ - n_i (1/2 + \alpha_i) }$.
 
Summing \eqref{momentseasy-first} over $\psi \in S'_{m,q} $,  \[ \sum_{\psi \in S'_{m,q}} \prod_{i=1}^r L(1/2 + \alpha_i, \psi)  = \sum_{\bfn \in \{0,\dots, m-1\}^r} Q( \bfn, \bfa)  \sum_{\substack{ f_1,\dots, f_r \in \mathbb F_q[T] \\ \textrm{monic} \\ \deg(f_i) = n_i }}  \sum_{\psi \in S'_{m,q}}  \psi( f_1f_2 \dots f_r).\]

 By Lemma \ref{average-of-characters}, splitting up the $q^m$ and $q^{m-1}$ parts, we obtain

\begin{equation}\label{momentseasy-split} \begin{split}  & q^m  \sum_{\bfn\in \{0,\dots, m-1\}}  Q( \bfn, \bfa) \Bigl| \Bigl\{    f_1,\dots, f_r \in \mathbb F_q[T] \textrm{ monic} \mid \deg(f_i) = n_i  , \prod_{i=1}^r f_i = T^{ n } +  O( T^{n - m -1} ) \Bigr \}\Bigr|  \\  &- q^{m-1}   \sum_{\bfn \in \{0,\dots, m-1\}^r}  Q( \bfn, \bfa) \Bigl|  \Bigl\{    f_1,\dots, f_r \in \mathbb F_q[T] \textrm{ monic} \mid \deg(f_i) = n_i  , \prod_{i=1}^r f_i = T^{ n } +  O( T^{n - m } ) \Bigr \}\Bigr| \end{split}   \end{equation}

For the first line of \eqref{momentseasy-split}, if $n \leq m$, this equation is only satisfied if $f_i = T^{n_i}$ for all $i$. Hence the cardinality in question is $1$, and the terms with $ n \leq m$ contribute \begin{equation}\label{me-split-top-small} q^m  \sum_{\substack{ \bfn \in \{0,\dots, m-1\}^r \\ n \leq m }}Q(\bfn, \bfa) = q^m  \prod_{i=1}^r \frac{ 1}{ 1- q^{ - (1/2 + \alpha_i) } } + O \left(  q^m m^{r-1}  q^{-m/2}  \right). \end{equation} 

Each term in the first line of \eqref{momentseasy-split} with $n > m$ is a special case of the sum handled in Theorem \ref{main} (where we take $c_1,\dots ,c_m=0$). Hence these terms contribute  \begin{equation}\label{me-split-top-big}  q^m \sum_{\substack{ \bfn \in \{0,\dots, m-1\}^r \\ n > m }} Q(\bfn, \bfa) \left(  q^{n-m} + O \left( (r+2)^{n+m}q^{\frac{1}{2}\left( n -m + \lfloor \frac{n }{p} \rfloor - \lfloor \frac{m}{p} \rfloor  +1\right)} \right) \right). \end{equation}

Similarly, for the second line of \eqref{momentseasy-split} , the terms with $n \leq m-1$ contribute  \begin{equation}\label{me-split-bottom-small} q^{m-1}   \prod_{i=1}^r \frac{ 1}{ 1- q^{ - n_i (1/2 + \alpha_i) }}  + O \left(  q^{m-1} m^{r-1}  q^{-m/2}  \right)\end{equation}  and the terms with $n>m-1$ contribute  \begin{equation}\label{me-split-bottom-big}   q^{m-1} \sum_{\substack{ \bfn \in \{0,\dots, m-1\}^r \\ n > m-1 } }Q(\bfn, \bfa) \left(  q^{n+1 -m} + O \left( (r+2)^{n+m-1}q^{\frac{1}{2}\left( n +1-m + \lfloor \frac{n }{p} \rfloor - \lfloor \frac{m-1}{p} \rfloor  +1\right)} \right)\right) . \end{equation}

The $q^m q^{n -m}$ and $q^{m-1} q^{n+1-m}$ terms from \eqref{me-split-top-big} and \eqref{me-split-bottom-big} cancel, except when $n$ is exactly $m$, in which case they do not cancel because \eqref{me-split-top-big} is a sum over $n>m$ and \eqref{me-split-bottom-big} is a sum over $n> m-1$. These uncancelled terms contribute $O( m^{r-1} q^{m/2})$. This leaves only the main terms of \eqref{me-split-top-small} and \eqref{me-split-bottom-small}, which combine to give the main term of \eqref{momentseasy-main}, and the error terms of all formulas. The error terms from \eqref{me-split-top-big} and \eqref{me-split-bottom-big} are of size

\begin{equation}\begin{split} & q^m \sum_{\bfn \in \{0,\dots, m-1\}^r} O \left( Q(\bfn, \bfa)(r+2)^{n +m}q^{\frac{1}{2}\left( n -m + \lfloor \frac{n }{p} \rfloor - \lfloor \frac{m}{p} \rfloor  +1\right)} \right) \\=& \sum_{\bfn \in \{0,\dots, m-1\}^r} O \left(  (r+2)^{n +m}q^{\frac{1}{2}\left( m + \lfloor \frac{n }{p} \rfloor - \lfloor \frac{m}{p} \rfloor  +1\right)} \right) \\ =& O \left( m^r  (r+2)^{ r(m-1)+m} q^{\frac{1}{2}\left( m + \lfloor \frac{ r (m-1)}{p} \rfloor - \lfloor \frac{m}{p} \rfloor + 1\right)} \right).\end{split}\end{equation}

Since the error terms arising from \eqref{me-split-top-small} and \eqref{me-split-bottom-small} are $O(q^{m/2} m^{r-1} )$, they are also \[O \left( m^r  (r+2)^{ r(m-1)+m} q^{\frac{1}{2}\left( m + \lfloor \frac{ r (m-1)}{p} \rfloor - \lfloor \frac{m}{p} \rfloor + 1\right)} \right),\] so we obtain \eqref{momentseasy-main}.
\end{proof}

We can prove also a similar estimate for a moment twisted by a positive power of $\epsilon_\psi$.

\begin{corollary}\label{momentsmedium} Let $m \geq 1, r \geq 0,$ and $s\geq 1$ be natural numbers, and let $\alpha_1,\dots,\alpha_r$ be complex numbers with nonnegative real part. Then

 \begin{equation}\label{mm-main} \frac{1}{  \left|S'_{m,q} \right| }  \sum_{\psi \in S'_{m,q} } \epsilon_\psi^s \prod_{i=1}^r L(1/2 + \alpha_i, \psi) = O \left( m^r (r+s+2)^{ (r+s) (m-1)+m} q^{\frac{1}{2}\left( -m + \lfloor \frac{ (r+s) (m-1)}{p} \rfloor - \lfloor \frac{m}{p} \rfloor + 1\right)} \right). \end{equation}

\end{corollary}

To obtain Theorem \ref{momentsmediumintro}, we again use Lemma \ref{chi-psi-comparison} to replace $L(1/2+\alpha_i,\psi)$ with $L(1/2+\alpha_i,\chi)$, which again introduces an $ O(1)$ factor.

\begin{proof} We use \eqref{root-number-formula} to express $\epsilon_\psi$ as a sum over polynomials of degree $m-1$. To that end, it is convenient to write $n_{r+1}, \dots, n_{r+s}= m-1$ and $n = \sum_{i=1}^{r+s} n_i$. We also set $\bfn = (n_1,\dots, n_r), \bfa= (\alpha_1,\dots,\alpha_r), \bff = (f_1,\dots, f_{r+s})$, and $Q(\bfn,\bfa) =   q^{- s (m-1)/2}\prod_{i=1}^r q^{ -n_i (1/2+\alpha_i) } $.

It follows from \eqref{root-number-formula} that \begin{equation}\label{mm-first}  \epsilon_\psi^s \prod_{i=1}^r L(1/2 + \alpha_i, \psi)  =\sum_{\bfn \in \{0,\dots, m-1\}^r}Q(\bfn, \bfa)  \sum_{\substack{ \bff \in \mathbb F_q[T]^{r+s}  \\ f_i \textrm{ monic} \\ \deg(f_i) = n_i }} \psi( f_1f_2 \dots f_{r+s}).\end{equation}

Summing over $\psi$ using Lemma \ref{average-of-characters}, we obtain \begin{equation}\label{mm-double} \begin{split} &\sum_{\psi \in S'_{m,q} }\epsilon_\psi^s \prod_{i=1}^r L(1/2 + \alpha_i, \psi) \\ = \sum_{\bfn \in \{0,\dots, m-1\}^r}  Q(\bfn, \bfa)       \Biggl( &  q^m\Bigl|  \Bigl\{    \bff \in \mathbb F_q[T] ^{r+s} \mid f_i\textrm{ monic}, \deg(f_i) = n_i  , \prod_{i=1}^{r+s} f_i = T^{ n } +  O( T^{n - m-1 } ) \Bigr \}\Bigr|  \\- & q^{m-1} \Bigl|  \Bigl\{    \bff \in \mathbb F_q[T]^{r+s}  \mid f_i\textrm{ monic}, \deg(f_i) = n_i  , \prod_{i=1}^{r+s} f_i = T^{ n } +  O( T^{n - m } ) \Bigr \}\Bigr| \Biggr) .\end{split}\end{equation}

We have \begin{equation}   \Bigl|  \Bigl\{    \bff \in \mathbb F_q[T]^{r+s} \mid f_i\textrm{ monic},  \deg(f_i) = n_i  , \prod_{i=1}^{r+s} f_i = T^{ n } +  O( T^{n - m } ) \Bigr \}\Bigr|  = q^{ n + (s-1) (m-1)}   \end{equation}  
because there is always exactly one value of $f_{r+s}$ that satisfies the equation for any $f_1,\dots,f_{r+s-1}$. Hence we can simplify \eqref{mm-double} 
\begin{equation}\label{mm-simplified} \hspace{-.5in} \begin{split} \sum_{\psi \in S'_{m,q} }&\epsilon_\psi^s \prod_{i=1}^r L(1/2 + \alpha_i, \psi) \\ = \sum_{\bfn \in \{0,\dots, m-1\}^r} & Q(\bfn, \bfa)       \Biggl(  q^m\Bigl|  \Bigl\{    \bff \in \mathbb F_q[T]^{r+s} \mid f_i\textrm{ monic},  \deg(f_i) = n_i  , \prod_{i=1}^{r+s} f_i = T^{ n } +  O( T^{n - m-1 } ) \Bigr \}\Bigr|  -  q^{n + s (m-1)} \Biggr) .\end{split}\end{equation}

The term in parantheses in \eqref{mm-simplified} is $q^m$ times the left side of Theorem~\ref{main} with $k=r+s$ and $c_1,\dots,c_m=0$. Applying Theorem~\ref{main}, we see that in each term with $n + s(m-1) \geq m$, we have \begin{equation}\label{main-individual} \begin{split} &\Biggl|  q^m\Bigl|  \Bigl\{    \bff \in \mathbb F_q[T]^{r+s} \mid f_i\textrm{ monic},  \deg(f_i) = n_i  , \prod_{i=1}^{r+s} f_i = T^{ n } +  O( T^{n - m-1 } ) \Bigr \}\Bigr|  -  q^{n + s (m-1)}  \Biggr|\\ \leq & 3 (r+s+2)^{n + s(m-1) +m}q^{\frac{1}{2}\left( n + s(m-1) +m + \lfloor \frac{n + s (m-1)}{p} \rfloor - \lfloor \frac{m}{p} \rfloor  +1\right)} . \end{split} \end{equation}

The product of \eqref{main-individual} with $Q(\bfn, \bfa) $ is at most \[ 3 (r+s+2)^{n + s(m-1) +m} q^{\frac{1}{2}\left( m + \lfloor \frac{n + s (m-1)}{p} \rfloor - \lfloor \frac{m}{p} \rfloor  +1\right)}. \] Summing over $\bfn$, we get \begin{equation}\label{mm-error} 3 m^r (r+s+2)^{r+s (m-1) +m}q^{\frac{1}{2}\left(  m + \lfloor \frac{ (r+s) (m-1)}{p} \rfloor - \lfloor \frac{m}{p} \rfloor  +1\right) }. \end{equation}
 
 The only term where $ n+ s(m-1) < m$ occurs when $s=1$ and $n_1,\dots,n_r=0$. This term is simply $q^{ (-m-1)/2} (q^m - q^{m-1})$ and is bounded by  \eqref{mm-error}. Thus, we obtain  \eqref{mm-main}.

\end{proof}

\bibliographystyle{plainnat}

\bibliography{references}

 \end{document}